\documentclass{amsart}

\usepackage{graphicx}
\usepackage{amssymb}
\usepackage{amsthm}

\usepackage{amsmath}

\usepackage{amsfonts}
\usepackage{amssymb}

\usepackage{mathtools}

\usepackage{subcaption}

\newtheorem{theorem}{Teorema}[section]
\newtheorem{lemma}[theorem]{Lemma}
\newtheorem{proposition}[theorem]{Proposition}

\theoremstyle{definition}

\begin{document}

\title{Satellite knots and trivializing bands}

\author{Lorena Armas-Sanabria}
\address{ \hskip-\parindent
Lorena Armas-Sanabria\\
Universidad Tecnol\'ogica de Quer\'etaro\\
Av. Pie de la Cuesta 2501\\
  Quer\'etaro, MEXICO}
 \email{lorenaarmas089@gmail.com}

\author{Mario Eudave-Mu\~noz}
\address{ \hskip-\parindent
Mario  Eudave-Mu\~noz\\
Instituto de Matem\'aticas\\
 Universidad Nacional Aut\'onoma de M\'exico\\
Unidad Juriquilla, Quer\'etaro, 
 MEXICO}
 \email{mario@matem.unam.mx}

\maketitle

\begin{abstract}
We show an infinite family of satellite knots that can be unknotted by a single band move, but such that there is no band unknotting the knots which is disjoint from the satellite torus.\end{abstract}

\section{Introduction}.

Let $K$ be a knot or link in the 3-sphere. A banding $K_b$ of $K$ is a knot or link obtained from $K$ by the following construction.
Let $b:I\times I \rightarrow S^3$ be an embedding such that $b(I\times I)\cap K = b(\partial I \times I)$, and then $K_b$
is defined as $K_b=(K-b(\partial I \times I))\cup b(I\times \partial I)$.

If $K$ and $K_b$ are both knots, a band move is also called an $H(2)$-move, see \cite{HN}. Then the $u_2$ unknotting number of a knot $K$,
$u_2(K)$, is defined as  the minimal number of $H(2)$ moves needed to transform $K$ into a trivial knot.

Let $V$ be a standard solid torus contained in $S^3$, and $K'$ a knot embedded in $V$ such that
$K'$ is not contained in any $3$-ball contained in $V$ nor it is isotopic to a longitude of $V$. Let $J$ be a non-trivial knot in $S^3$ and let $N(J)$ be a closed regular neighborhood of $J$. Let $h: V \rightarrow N(J)$ be a homeomorphism such that a preferred longitude $\lambda \subset V$ is mapped to a longitude $l$ of $N(J)$.
Then $h(K') = K$ is called a satellite knot with companion  knot $J$ and pattern $(V, K',l)$. The torus $Q=\partial N(J)$ is called a satellite torus of $K$. For an example of  a satellite knot see Figure \ref{satellite}.

It is easy to construct examples of satellite knots which can be unknotted by a single band move. Let $J$ be a non-trivial knot with neighborhood $N(J)$.
Take a trivial knot $U$ inside $N(J)$. Take a band $b$ for $U$ disjoint from the torus $\partial N(J)$, which wraps around $N(J)$ in a complicated manner.
The knot $K_b$ obtained by the band move  is a satellite knot with companion $J$. By taking a band $b'$ dual to $b$, we get the trivial knot $U$, and the band $b'$ is clearly disjoint from the companion torus.

So, it is natural to ask: 

\textit{Question:}
if $K$ is a satellite knot with satellite torus $Q$, which can be unknotted by a single band move, is there an isotopy that makes the corresponding band disjoint from $Q$?

Note that if instead of doing a band move we do a crossing change, then the answer to the previous question is positive. That is, any crossing change that unknots
a satellite knot can be made disjoint from the satellite torus. This was proved by Scharlemann and Thompson \cite{ST}. This result can be generalized to any
non-integral tangle replacement, see \cite{E}.

Here we give a negative answer to that question. We show the existence of an infinite family of knots $K(m,n,p;q)$, where $m,\, n,\, p\, ,q$ are integral parameters with some restrictions, such that $K(m,n,p;q)$ is a satellite knot with companion a $(2,-q)$-torus knot $J$, and a banding of it produces the trivial knot. It is shown that the band intersects the satellite torus $Q$ in two arcs and as the winding number of $K(m,n,p;q)$ in $N(J)$ is $\pm 1$ or $\pm 3$, it follows easily that the band cannot be made disjoint from the torus $Q$.

By taking double branched covers, we get an infinite family of strongly invertible hyperbolic knots $\tilde K(m,n,p;q)$, having an integral exceptional Dehn surgery that produces a manifold containing an incompressible torus $T$ that intersects the surgered solid tori in $4$ disks. In other words, there is an essential four punctured torus properly embedded in the exterior of $\tilde K(m,n,p;q)$. Many examples of knots with this type of surgery are already known, see the introduction of \cite{ER} for  a survey of this topic. The new in these examples is that the essential torus is disjoint from the involution axis of the strongly invertible knot.

%These examples can be generalized to a family of knots $K=K(a_1,a_2,\dots,a_n,p;q)$, where $a_1,a_2,\dots,a_n,n$ and $q$ are odd integers, and $p$ an even integer, and such
%that $K$ is a satellite knot with companion a $(2,q)$-torus knot,  a banding of it produces the trivial knot, and the winding number of $K$ in $N(J)$ is $\pm n$.

\section{ Main examples}.
\label{sec2}

Let $V$ be a standard solid torus in $S^3$ and $K'$ be the pattern shown in Figure \ref{pattern}, where $m,\, n,\, p,\, q$ denote integral numbers. An horizontal box labeled by, 
say $[m]$, denotes $m$ horizontal crossings, and the box labeled $[q]$ denotes $q$ vertical crossing. Our convention on the sign of the crossings is given in Figure \ref{pattern}. Assume that $m,\, n \not=0$, and that $q$ is odd; $p$ can be any integer. By inspection it can be seen that $K'$ is in fact a knot if $m, \, n$ are odd number and $p$ is even, or if $m$ is odd and $n,\, p$ are even, or if $m,\, p$ are even and $n$ odd, or if $m,\, n,\, p$ are all even. In the remaining cases $K'$ is a two components link.  The winding number of $K'$ in $V$ is $\pm 1$ or $\pm 3$, depending on an orientation given to $K'$. In fact, when $K'$ is a knot, it has winding number $\pm 3$ when $m,\, n$ are odd and $p$ is even, and it has winding number $\pm 1$ in all the other cases.

\begin{figure}

% \begin{center}
\includegraphics[angle=0, width=8true cm]{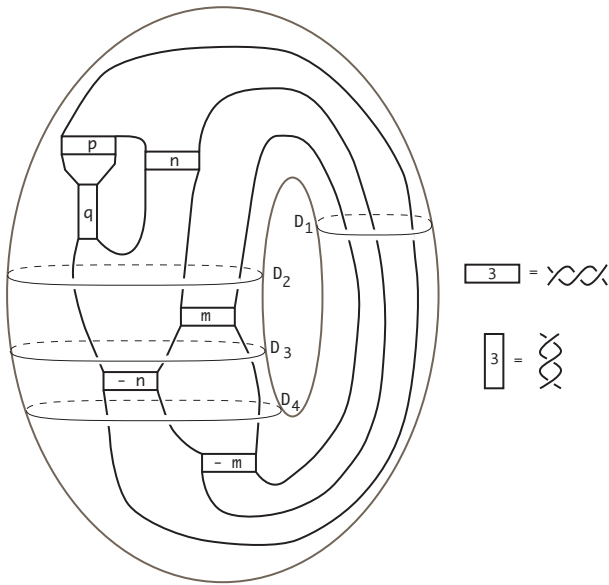}

\caption{The pattern knot.}
\label{pattern}
%\end{center}

\end{figure}

Consider the disks $D_1$, $D_2$, $D_3$, $D_4$ contained in $V$ as shown in Figure \ref{pattern}. These are meridian disks of $V$, each intersecting $K'$ in 3 points. Note that 
these disks divide $V$ in four regions which we denote by $H_1, H_2, H_3, H_4$, where $H_i$ is bounded by $D_i$ and $D_{i+1}$, $\textrm{mod}\, 4$. Note that $t_i=H_i\cap K'$ 
consists of three arcs, and then $(H_i,t_i)$ can be considered as a 3-tangle, as shown in Figure \ref{regions} (a) for $H_2$. The pair $(H_i,t_i)$ can be seen as the union of a 2-tangle,
in fact a rational tangle $(B,t)=R(p/q)$, and a trivial arc, as shown in Figure \ref{regions} (b).  In fact, $H_1$ can be seen as the union of the rational tangle $R((npq+n+p)/(pq+1))$ and a 
trivial arc by the right; $H_2$ can be seen as the union of the rational tangle $R(m)$ and a trivial arc by the left; $H_3$ can be seen as the union of the rational tangle $R(-n)$ and a
trivial arc by the right; $H_4$ can be seen as the union of the rational tangle $R(-m)$ and a trivial arc by the left. The tangle $(H_i,t_i)$ will be a braid, that is, it consists of descending 
arcs going from $D_i$ to $D_{i+1}$, only if $m=\pm 1$, or $n=\pm 1$, or $p=0$ and $n=\pm 1$, or $p=1,q=-3,n=1$, or $p=-1,q=3, n=-1$.

\begin{figure}

% \begin{center}
\includegraphics[angle=0, width=10true cm]{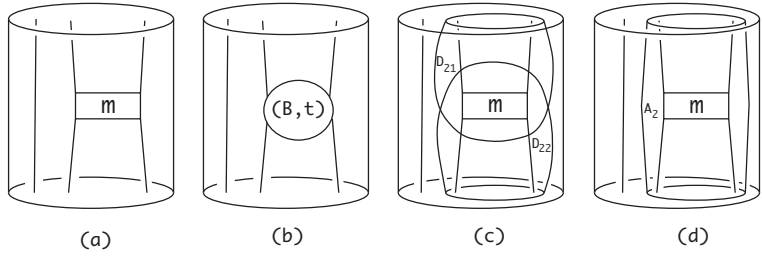}

\caption{The tangle $H_2$.}
\label{regions}
%\end{center}

\end{figure}

Note that there is a disk $D_{i1}$ in $H_i$, whose boundary is contained in $D_i$, and which intersects $K'$ in two
points, in fact, $D_{i1}$ cuts off the rational tangle $(B,t)$, as in Figure \ref{regions} (c). Note that  $D_{i1}$ is not isotopic to a disk in $D_i$, except when $(H_i,t_i)$ is a braid.  Similarly 
there is such a disk $D_{i2}$ whose boundary is contained in $D_{i+1}$, also shown in Figure \ref{regions} (c). Note also that there is an annulus $A_i$ in $H_i$, with one boundary 
component in $D_i$ and other in $D_{i+1}$, which encloses the tangle $(B,t)$, as in Figure \ref{regions} (d). If $(H_i,t_i)$ is a braid then there are several annuli which enclose two of 
the arcs of the braid, which we also denote by $A_i$.

\begin{lemma} \label{wrapping} The wrapping number of $K'$ in $V$ is exactly 3. \end{lemma}

\begin{proof} The winding number of $K'$ is $V$ is $\pm 1$ or $\pm 3$, so the wrapping number is 1 or 3. Suppose that it is 1. Then there is a meridian disk $D$ of $V$ which 
intersects $K'$ in one point. We can assume that $\partial D$ is disjoint from $\partial D_i$ for all $i$, hence $D\cap D_i$ consists only of simple closed curves. Let $\gamma$ be an 
innermost such curve in $D$. Then $\gamma$ bounds a disk $D'$ in $D$ with interior disjoint from the $D_i's$, and it is either disjoint from $K'$ or intersects it in one point.  The disk $D'$ must be in a region $H_i$. Then in fact $D'$ is isotopic to a disk in $D_i$, and after an isotopy the intersection is removed. We continue 
until there is no intersection left between $D$ and the $D_i's$. Then $D$ would be contained in a region $H_i$, which clearly is not possible. \end{proof}

\begin{lemma} \label{localknot} The knot $K'$ has no local knots, that is, if $S$ is a sphere in $V$ which intersects $K'$ in two points, then $S$ bounds a 3-ball which contains an unknotted spanning arc of $K'$.
\end{lemma}

\begin{proof} The proof is similar to that of Lemma \ref{wrapping}, by looking at the intersections between a sphere with the disks $D_i$. \end{proof}

\begin{lemma} \label{disks} If $E$ is a disk properly embedded in $H_i$, whose boundary lies in $D_i$ or $D_{i+1}$ and which intersects $K'$ in two points, then it is either parallel to a disk on $D_i$ or $D_{i+1}$, or it is isotopic to $D_{i1}$ or $D_{i2}$. \end{lemma}

\begin{proof} Suppose that the 3-tangle $(H_i,t_i)$ is not a braid, for in this case the conclusion is obvious. Let $E'$ be a disk in $H_i$, such that 
$\partial E'=\alpha \cup \beta$, where $\alpha$ is an arc in $D_i$ and $\beta$ is an arc in $D_{i1}$. Suppose that the interior of $E'$ is disjoint from $D_{i1}$, and that
$E'$ is disjoint from $K'$ or intersects it transversely in one point. The arc $\beta$ in $D_{i1}$ cuts off a disk $D'$, which is disjoint from $K'$ or intersects it in one point. 
If $E'$ is disjoint from $K'$ and $D'$ intersects $K'$ in one point, then the disk $E'$ would separate the strings of the tangle $(B,t)$ determined by $D_{i1}$, which is not 
possible. In all other cases, $E'$ and $D'$ are parallel, that is, cobound with part of $D_i$ a 3-ball disjoint from $K'$, or intersecting it in an unknotted spanning arc.

Suppose $E$ is a disk as in the statement of the lemma, whose boundary lies on $D_i$. Look at the intersections between $D_{i1}$ and $E$. Let $\gamma$ be a simple closed curve of 
intersection which is innermost in $E$. If $\gamma$ bounds a disk disjoint from $K'$ or intersecting it in one point, then this is easily removed, for there are no local knots in $K'$. Now 
suppose that $\gamma$ is an outermost arc of intersection in $E$, bounding a disk $E'$. Then, by the previous paragraph, this disk would be parallel to a disk in $D_{i1}$, and by an 
isotopy the intersection could be removed. The only remaining case is that $\gamma$ bounds a disk intersecting $K'$ twice, and that there are no arcs of intersection. In that case 
$\partial E$ must be parallel to $\partial D_{i1}$, and then $E$ is parallel to a disk in $D_i$, or it is parallel to $D_{i1}$. \end{proof}

\begin{lemma} \label{annulus} If $E$ is an annulus properly embedded in $H_i$, whose boundary lies in $D_i \cup D_{i+1}$, and such that it is disjoint from $K'$, then it is either parallel to an annulus on $D_i$ or $D_{i+1}$, or it is isotopic to the annulus $A_i$, or it is parallel to the part of $\partial V$ lying in $H_i$, or is parallel to the frontier of an arc of $t_i$. \end{lemma}

\begin{proof} Again, look at the intersections between $E$ and $D_{i1}$, $D_{i2}$ and $A_i$. \end{proof}

\begin{lemma}  \label{Conwaysphere}There is no Conway sphere in $V$. \end{lemma}

\begin{proof} Suppose that $S$ is a Conway sphere for $K'$ in $V$, that is, $S$ is a sphere in $V$ that intersects $K'$ transversely in four points, and $S-K'$ is incompressible in 
$V-K'$.  Look at the intersections between $S$ and the disks $D_i$. If this intersection is empty, then $S$ would be contained in some $H_i$, but this is not possible for $(H_i,t_i)$ is a 
trivial tangle. Let $\gamma$ be a simple closed curve of intersection which is innermost in $S$. If $\gamma$ bounds a disk disjoint from $K'$ or intersecting it in one point, then this is 
easily removed. So $\gamma$ bounds a disk intersecting $K'$ twice. It follows that the intersection consists of concentric curves, with two of them, say $\gamma_1$ and $\gamma_2$, 
bounding disks $E_1$ and $E_2$ which intersect $K'$ in two points. Between $E_1$ and $E_2$ there is a collection of annuli 
$F_1, \dots, F_r$, which may be empty. If one of $E_i$ or $F_j$ is parallel to a disk or annulus in some $D_i$, then the number of curves of intersection could be reduced. Then, say, 
$E_1$ is in some $H_i$ and by Lemma \ref{disks} is parallel to the disk $D_{i2}$, and similarly $E_2$ is in some $H_j$ and is parallel to the disk $D_{jt}$. The annuli $F_j$ are then 
spanning annuli in some $H_i$. Note that $\partial D_{i2}$ and $\partial D_{(i+1)1}$ have non-empty intersection, and also $\partial D_{i2}$ and $\partial A_{i+1}$ have non-empty 
intersection, except perhaps if $H_{i+1}$ is a braid. By inspection it follows that is not possible to assemble the pieces to get the desired sphere.\end{proof}

\begin{lemma} \label{torus} If $T$ is an incompressible torus in $V-K'$, then it is either parallel to $\partial V$, or parallel to $\partial N(K')$. \end{lemma}

\begin{proof} Let $T$ be an incompressible torus in $V-K'$. Look at the intersection between $T$ and the disks $D_i$. This intersection consists of curves that divide $T$ into annuli, 
which we can suppose are not annuli parallel to an annulus lying in some $D_i$, that is, all are spanning annuli in the $H_i$. If in some $H_i$ there is an annulus which is parallel to an 
annulus $A_i$, then it can be seen that the torus cannot be assembled. Then each such annuli is either an annulus running along an arc of $t_i$, or is parallel to $\partial V$. We 
conclude that the torus is parallel to $\partial V$, or parallel to $\partial N(K')$. \end{proof}

\begin{lemma} \label{annulus2} If $A$ is an annulus properly embedded in $V$, disjoint from $K'$ and incompressible in $V-K'$ then $A$ is parallel to an annulus in $\partial V$. There is no properly embedded M\"obius band in $V$ disjoint from $K'$.\end{lemma}

\begin{proof} Let $A$ be an incompressible annulus in $V-K'$. Look at the intersection between $A$ and the disks $D_i$. This intersection consists of arcs, which we can assume divide 
the annulus into rectangles. That is, the intersection of $A$ with each $H_i$ consists of one of more rectangles. But these rectangles cannot be assembled to form an annulus, except 
perhaps if each tangle  $H_i$ is a braid, that is, if $K'$ is a closed braid in $V$ and that $K'$ is parallel to a curve on $\partial V$. $K'$ is a closed braid only when $m=\pm 1$, $n=\pm 
1$ and $p=0$, or when  $m=\pm 1$, and  $p=1,q=3,n=1$, or when $m=\pm 1$, and $p=-1,q=-3, n=-1$. This gives a total of 8 closed braids, and it can be checked that none of them is parallel to a curve contained in $\partial V$. If there is a M\"obius band disjoint from $K'$, then there is also an incompressible torus disjoint from $K'$, which is not possible.\end{proof}

Let $J$ be a $(2,-q)$-torus knot, and let $h: V \rightarrow N(J)$ be a homeomorphism such that a preferred longitude $\lambda \subset V$ is mapped to a longitude $l$ of $N(J)$ of 
slope $-2q$. Let $K=K(m,n,p;q)$ be the satellite knot $h(K')$, and let $Q=\partial N(J)$ be the satellite torus. See Figure \ref{satellite}, which shows the case $q=3$.

\begin{figure}

% \begin{center}
\includegraphics[angle=0, width=7true cm]{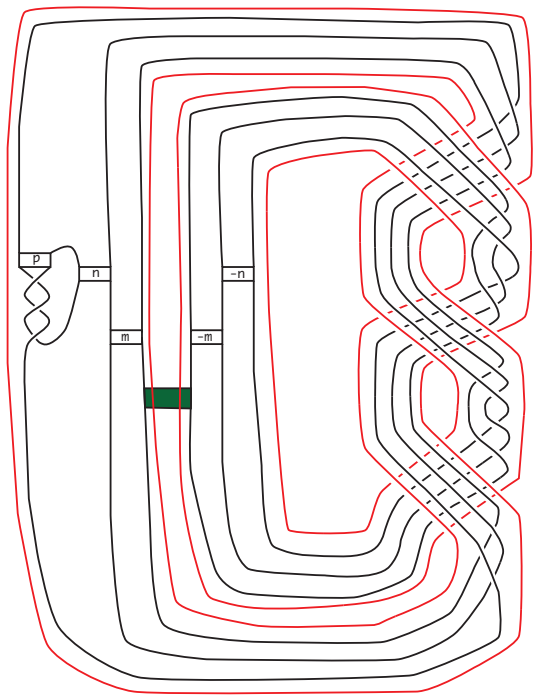}

\caption{The satellite knot.}
\label{satellite}
%\end{center}

\end{figure}
 
Let $b:I\times I \rightarrow S^3$ be the banding of $K(m,n,p;q)$ shown in Figure \ref{satellite}, an let $K_b$ be the knot so obtained.

\begin{proposition} \label{trivialknot} The knot $K_b$ is a trivial knot.
\end{proposition}
\begin{proof} To see this in an easy way, note that there is a M\"obius band $\mathcal{M}$ properly embedded in the exterior of $J$ with slope $-2q$. Let $A$ be an annulus embedded in $V$ whose boundary consists of preferred longitudes of $V$, such that a diagram of $K'$ can be drawn on $A$. Consider now $h(A)$ and note that one of its boundary components coincide with $\partial \mathcal{M}$. Then $\mathcal{M}'=\mathcal{M}\cup h(A)$ is a M\"obius band which contains a diagram of $K(m,n,p;q)$. Consider the band $b$ shown in Figure \ref{satellite} and note that it is contained in $\mathcal{M}'$, and then we have a diagram of $K_b$ on $\mathcal{M}'$. Following Figure \ref{unknot}, do isotopies of $K_b$ along $\mathcal{M}'$ to undo the crossing in the boxes $m$, $-m$, then undo the crossings in the boxes $n$ and $-n$, until we get to the penultimate diagram of Figure \ref{unknot}. There the crossings in the box labelled with $q$ cancel with the crossings introduced by the M\"obius band, and then just undo the crossings in the box labeled by $p$.
\end{proof}

\begin{figure}

% \begin{center}
\includegraphics[angle=0, width=10true cm]{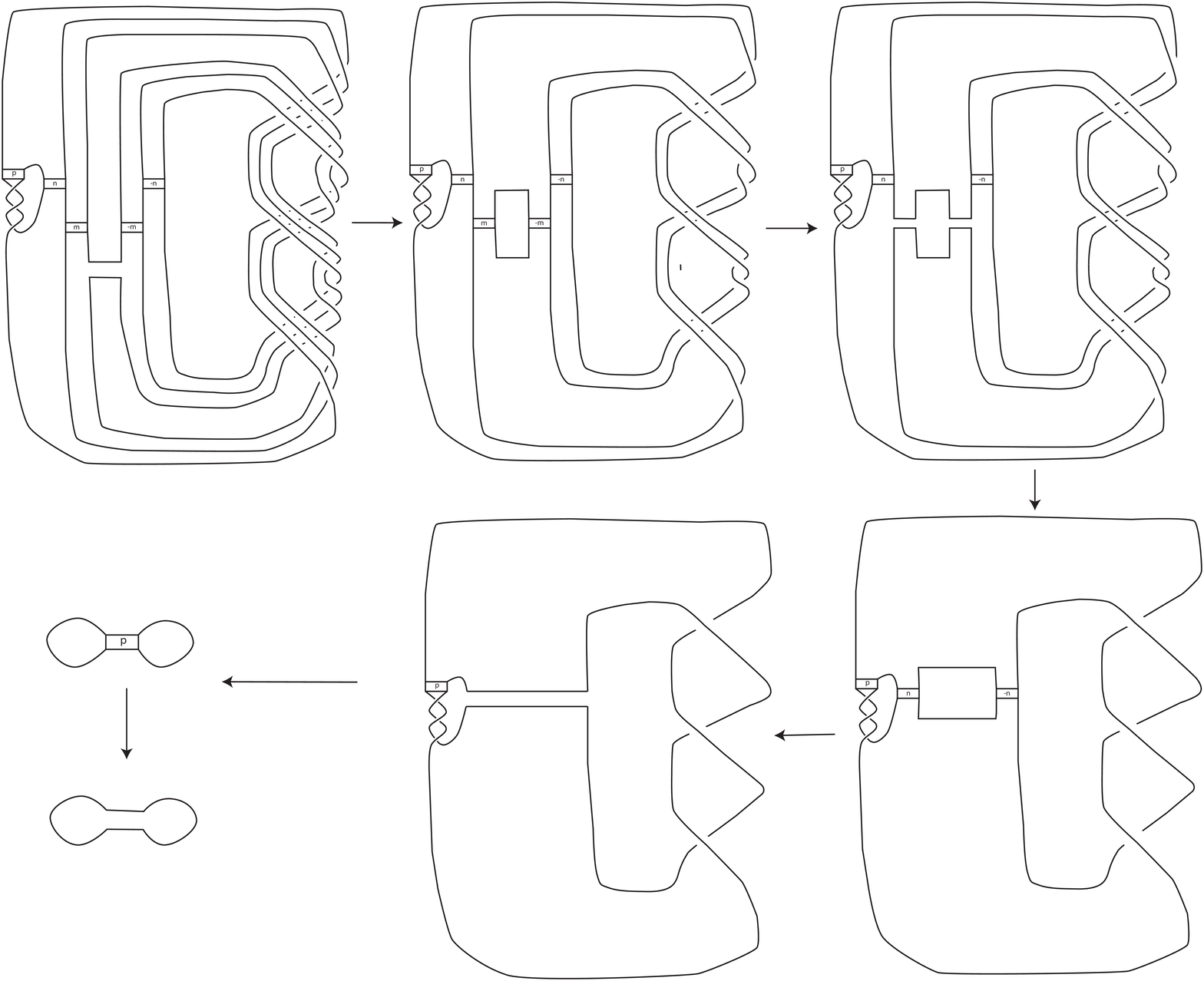}

\caption{Obtaining the trivial knot.}
\label{unknot}
%\end{center}

\end{figure}

Now we show that the band cannot be isotoped to be disjoint from $Q$, in fact we can show more.

\begin{proposition} \label{bandnodisjoint} No banding of $K(m,n,p;q)$ that produces a trivial knot can be disjoint from $Q$.
\end{proposition}
\begin{proof}
Suppose that $b$ is a band for  $K=K(m,n,p;q)$ which produces by banding a trivial knot $U$ and that $b$ is disjoint from $Q$. Give an orientation to $K$. Remember that  the wrapping number of $K$ in the solid torus $h(V)$ is $3$. Then $b$ intersects a meridian 
disk of $h(V)$ in a collection of $r$ arcs. So $U$ intersects a meridian disk of $h(V)$ in $3 + 2r$ points. Then the winding number of $U$ in $h(V)$ is an odd number. and then $U$ cannot be inside a $3$-ball contained in $h(V)$. It follows that $Q$ is incompressible in the exterior of $U$, then $U$ is not trivial.
\end{proof}

\begin{lemma} \label{uniquetorus} The torus $Q$ is the only incompressible torus in the exterior of $K$, and there is no Conway sphere for $K$.
\end{lemma}

\begin{proof} Suppose that $P$ is a Conway sphere for $K$. Consider 
first the intersections between $P$ and $Q$. We can assume that the intersection consists of curves that are essential on both $P$ and $Q$. Let $\gamma$ be an 
innermost curve of intersection in $P$.  Then $\gamma$ bounds a disk $D$ in $P$ containing one or two points of intersection with $K$. Because the wrapping number 
of $K$ in $h(V)$ is $3$, this is not possible. Then $P$ is disjoint from $Q$, and it must be contained in $h(V)$. By Lemma \ref{Conwaysphere} this is not possible.

Suppose now that $P$ is an essential torus not isotopic to $Q$. Again look at the intersection between $P$ and $Q$, which then consists of curves which are essential in both $P$ and 
$Q$. Let $\gamma_1$ and $\gamma_2$ be a pair of intersection curves that bound an annulus $A$ in $P$ with interior disjoint from $Q$ and disjoint from $N(J)$. If $A$ 
is a boundary parallel annulus in the exterior of $J$, $E(J)$, an isotopy reduces the number of arcs of intersection. If $A$ is essential in $E(J)$, then its boundary consists of curves of slope $-2q$ on $Q$.
Then the annulus $A'$ in $P$ adjacent to $A$ is boundary parallel in $h(V)$. By an isotopy the number of arcs of intersection are reduced. Then $P$ is disjoint from $Q$.
As $J$ is a torus knot, there are not essential tori in its complement. Then $P$ is contained in $V$. By Lemma \ref{torus} this is not possible. \end{proof}

\section{ Dehn surgery}.
\label{sec3}

Let $K=K(m,n,p;q)$ be the knot or link defined in the previous section and let $b$ the band defined in Proposition \ref{trivialknot}. Let $B'=N(b)$. Then $B'$ is a 3-ball that intersects $K$  in two arcs. Let $K_b$ be the banding of $K$, so $K_b$ is the trivial knot; $K_b$ intersects $B'$ 
also in two arcs. Let $B$ be the complementary ball of $B'$ in $S^3$, and let $t=K\cap B$, so $t$ consists of two arcs properly embedded in $B$. That is, the pair $(B,t)$ 
can be considered as a 2-tangle, and the knots $K$ and $K_b$ can be considered as obtained by filling $B$ with rational tangles.

Let $\pi:S^3 \rightarrow S^3$ be the the double cover of $S^3$ branched along $K_b$. As $K_b$ is a trivial knot, the double branched cover is the 3-sphere.
Let $N=\pi^{-1}(N(b))$. Then $N$ is a solid torus, its core is a knot in $S^3$ which we denote by $\tilde K = \tilde K(m,n,p;q)$. By Montesinos trick, the double 
branched cover of $K$ is then obtained by performing $r$-Dehn surgery on $\tilde K$ for some integral slope $r$; this surgered manifold is denoted by $\tilde K(r)$.  
Let $\tilde \pi= \tilde K(r) \rightarrow S^3$ be the corresponding branched cover. Let $\tilde V=\tilde \pi^{-1}(V)$, and let $\tilde M=\tilde \pi^{-1} (E(J))$, and let
$T=\tilde \pi^{-1} (Q)$.
Note that because the winding number of $K$ in $h(V)$ is an odd number, $T$ is then a single torus which is  a double cover of $Q$. The torus 
$Q$ intersects $B'$ in two disks, which then implies that $T$ intersects $N$ in four disks.
Then $\tilde T = T \cap E(\tilde K)$ is a four-punctured torus properly embedded in the exterior of $\tilde K$.

\begin{lemma} \label{JSJ} The $JSJ$ decomposition of $\tilde K(r)$ is given by $\tilde K(r)=\tilde V \cup \tilde M$, where $\tilde V$ is a hyperbolic manifold and $\tilde M$ is a Seifert fibered space.
\end{lemma}

\begin{proof} The manifold $\tilde M$ is the double cyclic cover of the exterior of the $(2,-q)$ torus knot $J$, and then by \cite{RN} is the 
Seifert fibered manifold $(D^2; \frac{(q-1)/2}{q},\frac{(q-1)/2}{q})$.

If $\tilde V$ is not an hyperbolic manifold, then by Thurston Hyperbolicity Theorem, $\tilde V$ will have a compressing disk, a reducing sphere, an essential annulus or an essential 
torus. If there is any of these surfaces, by one of the several equivariant results \cite{KT}, \cite{K}, \cite{H}, there is a surface $S$ which is equivariant, that is, $\tau(S)=S$, or 
$\tau(S)\cap S= \emptyset$, where $\tau$ is the involution defined on $\tilde V$ by the double branched cover. By taking the projection $\tilde \pi(S)$, there will be a compression disk for $\partial V$ disjoint from $K'$, a meridian disk intersecting $K'$ in one point, a decomposing sphere for $K'$, a meridian disk intersecting $K'$ in two points, an essential annulus or M\"obius band, or an essential torus or Conway sphere for $K'$. None of these can exist by Lemmas \ref{wrapping}, \ref{localknot}, \ref{Conwaysphere}, \ref{torus} or \ref{annulus2},.  \end{proof}

\begin{lemma} \label{hyperbolic} The knot $\tilde K=\tilde K(m,n,p;q)$ is hyperbolic.
\end{lemma}

\begin{proof}
First note that $\tilde K$ cannot be a torus knot, for it has a Dehn surgery producing a 3-manifold containing a separating incompressible torus. Suppose $\tilde K$ is a 
satellite knot, and let $R$ be an incompressible torus in its exterior. As $\tilde K$ is strongly invertible, by the Equivariant Torus Theorem \cite{H}, there is an incompressible torus $\tilde R$, which is equivariant under the involution defined on $\tilde K$. $\tilde R$ bounds a solid torus $W$ which contains $\tilde K$. The complement of $W$ may contain incompressible tori other than $\tilde R$, but applying again the Equivariant Torus Theorem, there will be another equivariant torus in the complement of $W$. Therefore we can assume that there is an equivariant torus $R$, which defines then a companion knot $\overline K$ for $\tilde K$ which is hyperbolic or a torus knot. 

Suppose first that $R$ compresses after performing $r$-Dehn surgery. Then by \cite{Ga}, $\tilde K$ is a 0 or 1-bridge braid in $W$, with winding number $w\geq 2$. It follows from
\cite{Go} that the manifold $\tilde K(r)$ is obtained by $n^2$-Dehn surgery on $\overline K$, for some integer $n$. This is not possible if $\overline K$ is a torus knot, for the surgery produces a 3-manifold containing a separating incompressible torus. Also, this is not possible if $\overline K$ is an hyperbolic knot, because by \cite{GL}, we should have $w^2 \leq 2$.

Suppose now that $R$ is incompressible in $\tilde K(r)$. As $R$ is equivariant, $\pi(R)$, is a torus or a Conway sphere. But, by Lemma \ref{uniquetorus} there are no Conway spheres for $K$. If $\pi(R)$ is a satellite torus for $K$, it would have to be isotopic to $Q$ by Lemma \ref{uniquetorus}. But, as $\pi(R)$ is disjoint from the band, this is not possible by Proposition \ref{bandnodisjoint}. \end{proof}

So we have shown, 

\begin{theorem} There is a family of hyperbolic strongly invertible knots $\tilde K(m,n,p;q)$, which have a toroidal surgery containing a unique incompressible torus which hits the surgered solid torus in four disks. Furthermore, the incompressible torus is disjoint from the axis of the involution.
\end{theorem}

%Furthermore, tho tours $T$ divides $\tilde K(r)$ in two pieces, one is the Seifert fibered manifold, which is the double branched cover of the torus knot $(2,q)$. The other piece is the double cover of $V$ benched along $K'$, which is a hyperbolic manifold, by Lemma.

\section{ Concluding remarks}.
\label{sec4}

The knots $K(m,n,p;q)$ can be generalized to $K(a_1,a_2,$ $\dots,a_N,p;q)$, where $a_1,a_2,\dots, a_N$, $p$ and $q$ are integers, such that $N$ is even and $q$ is odd. A pattern for the case $N=4$ is shown in Figure \ref{generalizedexample}. In this case the wrapping number is $N+1$. The proofs of Sections 2 and 3 would be identical for large subfamilies of these knots. 

The examples can also be generalized when $N$ is odd and then the wrapping and winding number are even. But in this case the proof of Proposition \ref{bandnodisjoint} does not hold, for the winding number of $K'$ in $V$ is even, and a new proof would have to be done. Also, when taking double branched covers in this case, the torus $Q$ would lift to two tori, giving then examples of hyperbolic knots having two disjoint, non parallel, incompressible tori after Dehn surgery, each intersecting the surgered solid tori in two disks.

\begin{figure}

% \begin{center}
\includegraphics[angle=0, width=7true cm]{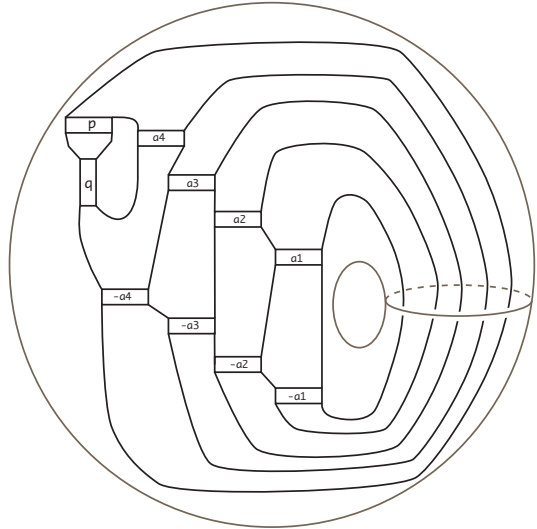}

\caption{The generalized pattern knot.}
\label{generalizedexample}
%\end{center}

\end{figure}

We finish with one question,

If $K$ is a satellite knot with satellite torus $Q$, and $b$ is a band that unknots $K$, is there a universal bound for the minimal number of arcs of intersection between $b$ and $Q$? 
Is this bound just two?

\vskip20pt
\textbf{Acknowledgements.} Research partially supported by grant PAPIIT-UNAM IN101317.

\end{document}